\documentclass{amsart}

\usepackage{fullpage}
\usepackage[utf8]{inputenc}
\usepackage[T1]{fontenc}
\usepackage[all]{xy}
\usepackage{graphicx}
\usepackage{amsmath,amsfonts,amssymb,mathtools}
\usepackage{stmaryrd}
\usepackage{graphicx}
\usepackage{mathrsfs,dsfont}
\usepackage{amsmath,amsthm,amsthm}

\usepackage{hyperref}

\usepackage{enumerate}

\usepackage{color}

\newcommand{\E}{\mathbb{E}}
\newcommand{\R}{\mathbb{R}}
\newcommand{\N}{\mathbb{N}}
\newtheorem{theo}{Theorem}[section]

\newtheorem{rem}[theo]{Remark}

\newtheorem{propo}[theo]{Proposition}
\newtheorem{lemma}[theo]{Lemma}

\newtheorem{ass}[theo]{Assumption}

\newcommand*\diff{\mathop{}\!\mathrm{d}}

\title{Averaging principle for slow-fast fractional stochastic differential equations}

\author{Charles-Edouard Br\'ehier}
              \address{Universite de Pau et des Pays de l'Adour, E2S UPPA, CNRS, LMAP,Pau, France}
              \email{charles-edouard.brehier@univ-pau.fr}
              
\author{Ibrahima Faye}       
             \address{Université Alioune Diop, Bambey, Sénégal}
              \email{ibou.faye@uadb.edu.sn}       
              
\begin{document}

\begin{abstract}
We prove the averaging principle for a class of stochastic systems. The slow component is solution to a fractional differential equation, which is coupled with a fast component considered as solution to an ergodic stochastic differential equation driven by a standard Brownian motion. We establish the convergence of the slow component when the time-scale separation vanishes to the solution of the so-called averaged equation, which is an autonomous fractional differential equation, in the mean-square sense. Moreover, when the fast component does not depend on the slow component, we provide a rate of convergence depending on the order of the fractional derivative.
\end{abstract}

\maketitle

{\small\noindent
{\bf AMS Classification.} 60H10, 34A08,  34C29.

\bigskip\noindent{\bf Keywords.} Stochastic differential equations. Fractional differential equations. Averaging principle.

\section{Introduction}
Multiscale modeling  and computation combined with stochastic analysis has quickly become a field of research that has a fundamental impact on all areas of science,  particularly physics, chemistry, finance and engineering. Many problems in these areas involve components that vary according to different time scales, for example slow and fast systems. These systems are often driven by standard or fractional Brownian motion associated with classical integer-order derivative~\cite{Pavliotis2008,Veretennikov1991, CEB2022}. However fractional derivatives are more appropriate in many cases, including nonlinear models. Fractional calculus is very useful in mechanic, chemistry, finance, biology and signal and image processing, see for instance~\cite{Kilbas2006, Samko1993}.

In this article, we consider the following class of slow-fast fractional stochastic differential equations, depending on the time-scale separation parameter $\epsilon$:
\begin{equation}\label{eq:system}
\left\lbrace
\begin{aligned}
X^\epsilon(t)&=x_0+\frac{1}{\Gamma(\alpha)}\int_{0}^{t}(t-s)^{\alpha-1}f\left(X^\epsilon(s),Y^\epsilon(s)\right) \diff s,\quad \forall~t\ge 0,\\
\diff Y^\epsilon(t)&=\frac{1}{\epsilon}b(X^\epsilon(t),Y^\epsilon(t))\diff t+\frac{1}{\sqrt{\epsilon}}\sigma(X^\epsilon(t),Y^\epsilon(t)) \diff B(t),\quad \forall~t\ge 0,\\
Y^\epsilon(0)&=y_0.
\end{aligned}
\right.
\end{equation}
We are interested in the regime $\epsilon\to 0$, i.e. when $X^\epsilon$ and $Y^\epsilon$ are respectively slow and fast components. The slow component $X^\epsilon$ evolves following a fractional differential equation of order $\alpha\in(0,1)$. The fast component $Y^\epsilon$ is solution to a stochastic differential equation driven by a standard Brownian motion $\bigl(B(t)\bigr)_{t\ge 0}$. We refer to Section~\ref{sec:setting} for precise assumptions. Let us mention that all the coefficients are assumed to be globally Lipschitz continuous, and that the fast equation with frozen slow component is ergodic.

The objective of this article is to establish the averaging principle: when $\epsilon\to 0$, the slow component $X^{\epsilon}$ converges to $\overline{X}$ in the mean square sense, where $\overline{X}$ is the solution of the averaged equation
\begin{equation}
\overline{X}(t)=x_0+\frac{1}{\Gamma(\alpha)}\int_{0}^{t}(t-s)^{\alpha-1}\overline{f}(\overline{X}(s))\diff s,\quad \forall~t\ge 0.
\end{equation}
where the averaged coefficient $\overline{f}$ is defined by~\eqref{eq:fbar}.

The averaging principle for stochastic differential equations has first been studied by Khasminskii in the seminal article~\cite{Khasminskii1968}. Since then, there have been many contributions on this topic, motivated by theory and applications. We do not intend to review the whole literature, the list of references below is not exhaustive. Many authors have weakened the conditions on the coefficients~\cite{Liu2020} or have considered variants of the evolution equations~\cite{CEB2022, Hairer2020}. In the literature, there are many contributions dealing with the behavior of multiscale systems driven either by standard brownian motion, fractional brownian motion, or both: see for instance~\cite{Freidlin1984,Veretennikov1991,Liu2020,Hairer2020, CEB2022}. However, to the best of our knowledge, there are no results for slow-fast systems where the slow component is solution to a fractional differential equation and the fast component is solution to a stochastic differential equation.

The first main result of this article is Theorem~\ref{theo:1}, which shows the averaging principle, in the following form: one has
\[
\underset{\epsilon\to 0}\lim~\underset{t\in[0,T]}\sup~\E[\|{X}^{\epsilon}(t)-\overline{X}(t)\|^2]=0.
\]
The proof is based on the introduction of auxiliary processes $\widehat{X}^{\epsilon,\delta}$ and $\widehat{Y}^{\epsilon,\delta}$, depending on the auxiliary parameter $\delta$, see Equation~\eqref{eq:aux} in Section~\ref{sec:aux}. The parameter $\delta$ may be interpreted as a time-step size, but the auxiliary processes cannot be simulated exactly. Proving appropriate error bounds and letting $\epsilon\to 0$ and $\delta\to 0$ provides the convergence results. However, the approach does not provide a speed of convergence with respect to $\epsilon$.

The second main result is Theorem~\ref{theo:2} and deals with the simpler situation, where the coefficients $b$ and $\sigma$, and thus the fast component $Y^\epsilon$, do not depend on the slow component. In that case, one obtains the following error estimates
\[
\underset{t\in[0,T]}\sup~\bigl(\E[\|{X}^{\epsilon}(t)-\overline{X}(t)\|^2]\bigr)^{\frac12}\le C_\alpha(T)\bigl(1+\|x_0\|+\|y_0\|\bigr)\epsilon^{\frac{\alpha}{2}},
\]
meaning that the order of convergence in the averaging principle for this kind of systems  is $\alpha/2$. It is not known whether this rate of convergence is optimal, and whether it can be obtained in the general situation considered in Theorem~\ref{theo:1}. These questions could be investigated in future works. The proof of Theorem~\ref{theo:2} is based on simpler arguments.

This article is organized as follows. Section~\ref{sec:setting} presents the main assumptions, the system and the averaging principle. The main results of this article, i.e. Theorem~\ref{theo:1} which justifies the averaging principle and Theorem~\ref{theo:2} which provides a rate of convergence in a specific situation, are stated in Section~\ref{sec:main}. Section~\ref{sec:aux} is devoted to providing some auxiliary results, such as moment bounds and regularity properties, on the solutions to the system and to the averaged equation. The proofs of further auxiliary results and then of the main results are given in Section~\ref{sec:proofs}.

\section{Setting}\label{sec:setting}

\subsection{Notation}

Let $p,q,m\in\mathbb{N}$ denote integers. Let $\mathcal{L}(\R^m,\R^q)$ denote the space of bounded linear operators from $\R^m$ to $\R^q$, which can be identified with the space $\mathcal{M}_{q,m}(\R)$ of matrices. The Euclidean norms in the spaces $\R^p$, $\R^q$ and $\mathcal{L}(\R^m,\R^q)$ are denoted by $\|\cdot\|$. The inner products in $\R^p$ and $\R^q$ are denoted by $\langle \cdot,\cdot \rangle$.

Let $\bigl(B(t)\bigr)_{t\ge 0}$ denote a standard $\R^m$-valued Brownian motion, defined on a probability space $(\Omega,\mathcal{F},\mathbb{P})$ which satisfies the usual conditions. Let $\bigl(\mathcal{F}_t\bigr)_{t\ge 0}$ denote the filtration generated by the Brownian motion.

Let $x_0\in\R^p$ and $y_0\in\R^q$, which are assumed to be deterministic (or $\mathcal{F}_0$ measurable).

Given $d\in\N$  and a Lipschitz continuous mapping $\phi:\R^d\to\R$, let
\[
{\rm Lip}(\phi)=\underset{z_1,z_2\in \R^d,z_1\neq z_2}\sup~\frac{|\phi(z_2)-\phi(z_1)|}{|z_2-z_1|}.
\]

Let $\alpha\in(0,1)$.

Without loss of generality it is assumed that the parameter $\epsilon$ takes values in $(0,1)$.

\subsection{Assumptions}

\begin{ass}\label{ass:F}
Assume that $f:\R^p\times \R^q\to \R^p$ is a globally Lipschitz continuous mapping: there exists $L_{f}\in(0,\infty)$ such that for all $x_1,x_2\in\R^p$ and all $y_1,y_2\in\R^q$ one has
\[
\|f(x_2,y_2)-f(x_1,y_1)\|\le L_{f}\bigl(\|x_2-x_1\|+\|y_2-y_1\|\bigr).
\] 
\end{ass}

\begin{ass}\label{ass:b-sigma}
Assume that $b:\R^p\times\R^q\to \R^q$ and that $\sigma:\R^p\times\R^q\to \mathcal{L}(\R^m,\R^q)$ are globally Lipschitz continuous mappings: there exist $L_{b},L_{\sigma}\in(0,\infty)$ such that for all $x_1,x_2\in\R^p$ and all $y_1,y_2\in\R^q$ one has
\begin{align*}
&\|b(x_2,y_2)-b(x_1,y_1)\|\le L_{b}\bigl(\|x_2-x_1\|+\|y_2-y_1\|\bigr)\\
&\|\sigma(x_2,y_2)-\sigma(x_1,y_1)\|\le L_{\sigma}\bigl(\|x_2-x_1\|+\|y_2-y_1\|\bigr)
\end{align*}
Moreover, the mappings $b(x,\cdot)$ and $\sigma(x,\cdot)$ satisfy the following dissipation property, uniformly with respect to $x\in\R^p$: there exists $\gamma\in(0,\infty)$ such that for all $x\in\R^p$ and all $y_1,y_2\in\R^q$ one has
\begin{equation}\label{eq:dissipation}
\langle b(x,y_2)-b(x,y_1),y_2-y_1\rangle+\frac12\|\sigma(x,y_2)-\sigma(x,y_1)\|^2\le -\gamma\|y_2-y_1\|^2.
\end{equation}
\end{ass}

Owing to the global Lipschitz continuity conditions given by Assumptions~\ref{ass:F} and~\ref{ass:b-sigma}, the mapping $f$, $b$ and $\sigma$ have at most linear growth: one has
\begin{equation}\label{eq:lineargrowth}
\underset{(x,y)\in\R^p\times\R^q}\sup~\frac{\|f(x,y)\|+\|b(x,y)\|+\|\sigma(x,y)\|}{1+\|x\|+\|y\|}<\infty.
\end{equation}
In addition, combining the dissipation property~\eqref{eq:dissipation} and the linear growth property~\eqref{eq:lineargrowth}, one obtains the following: there exists $C_{\gamma}\in(0,\infty)$ such that for all $(x,y)\in\R^p\times\R^q$ one has
\begin{equation}\label{eq:bounddissip}
\langle b(x,y),y\rangle +\frac12\|\sigma(x,y)\|^2\le -\frac{\gamma}{2}\|y\|^2+C_\gamma(1+\|x\|^2).
\end{equation}

\begin{propo}
Let $\alpha\in(0,1)$. For any $T\in(0,\infty)$ and $\epsilon\in(0,1)$, there exists a unique solution $\bigl((X^\epsilon(t),Y^\epsilon(t))\bigr)_{t\in[0,T]}$ to~\eqref{eq:system}.
\end{propo}

\subsection{Averaging principle}

Given arbitrary $x\in\R^p$, consider the stochastic differential equation for the fast process
\begin{equation}\label{eq:frozenSDE}
\diff Y^{x}(t)=b(x,Y^{x}(t))\diff t+\sigma(x,Y^{x}(t))\diff B(t),\quad \forall~t\ge 0,
\end{equation}
where the slow component $x$ is frozen. Owing to the Lipschitz continuous properties on $b$ and $\sigma$ given in Assumption~\ref{ass:b-sigma}, given an arbitrary initial condition $Y^x(0)$, the equation~\eqref{eq:frozenSDE} admits a unique solution $\bigl(Y^x(t)\bigr)_{t\ge 0}$. As a consequence of the dissipation property~\eqref{eq:dissipation}, the following result is standard.

\begin{propo}\label{propo:ergofrozenSDE}
For any $x\in\R^p$, the stochastic differential equation~\eqref{eq:frozenSDE} admits a unique invariant distribution $\mu^{x}$, which satisfies
\begin{equation}\label{eq:boundmux}
\underset{x\in\R^p}\sup~\frac{\int_{\R^q} \|y\|^2 \diff \mu^x(y)}{1+\|x\|^2}<\infty. 
\end{equation}
Moreover, there exists $C\in(0,\infty)$, such that for any Lipschitz continuous mapping $\phi:\R^q\to\R$, for all $t\ge 0$ and all $x\in\R^p$, one has
\begin{equation}\label{eq:ergofrozenSDE}
\big|\E[\phi(Y^x(t))]-\int_{\R^q} \phi(y) \diff\mu^x(y)\big|\le C{\rm Lip}(\phi)e^{-\gamma t}(1+\|x\|+\E[\|Y^{x}(0)\|])
\end{equation}
and for all $x_1,x_2\in\R^p$, one has
\begin{equation}\label{eq:Lipmux}
\big|\int_{\R^q} \phi(y) \diff\mu^{x_2}(y)-\int_{\R^q} \phi(y) \diff\mu^{x_1}(y)\big|\le C{\rm Lip}(\phi)\|x_2-x_1\|.
\end{equation}
\end{propo}

For any $x\in\R^p$, introduce the semigroup $\bigl(P_t^x\bigr)_{t\ge 0}$ defined by
\begin{equation}\label{eq:semigroupfrozenSDE}
P_t^x\phi(y)=\E[\phi(Y^x(t))|Y^x(0)=y],\quad\forall~y\in\R^q,
\end{equation}
if $\phi:\R^q\to\R$ is a bounded and continuous mapping.  If $\phi$ is a Lipschitz continuous mapping, the upper bound~\eqref{eq:ergofrozenSDE} can be written as
\begin{equation}\label{eq:ergosemigroupfrozenSDE}
\big|P_t^x\phi(y)-\int_{\R^q}\phi(y) \diff\mu^x(y)\big|\le C{\rm Lip}(\phi)e^{-\gamma t}(1+\|x\|+\|y\|),\quad \forall~x\in\R^p,y\in\R^q,t\ge 0.
\end{equation}

The averaged coefficient $\overline{f}:\R^p\to\R^p$ is defined by
\begin{equation}\label{eq:fbar}
\overline{f}(x)=\int_{\R^q} f(x,y) \diff \mu^x(y),\quad \forall~x\in\R^p.
\end{equation}
Note that $\overline{f}(x)$ is well-defined for all $x\in\R^p$ owing to the linear growth condition~\eqref{eq:lineargrowth} and to the bound~\eqref{eq:boundmux}. In addition, due to the Lipschitz continuity condition on $f$ (Assumption~\ref{ass:F}) and to the Lipschitz continuity property~\eqref{eq:Lipmux} of the invariant distribution $\mu^x$ with respect to $x$, the mapping $\overline{f}$ is globally Lipschitz continuous: there exists $\overline{L}_f\in(0,\infty)$ such that for all $x_1,x_2\in\R^p$ one has
\[
\|\overline{f}(x_2)-\overline{f}(x_1)\|\le \overline{L}_f\|x_2-x_1\|.
\]
As a consequence, $\overline{f}$ has at most linear growth: one has
\begin{equation}\label{eq:lineargrowth2}
\underset{x\in\R^p}\sup~\frac{\|\overline{f}(x)\|}{1+\|x\|}<\infty.
\end{equation}

The averaged equation is the fractional differential equation
\begin{equation}\label{eq:averaged}
\overline{X}(t)=x_0+\frac{1}{\Gamma(\alpha)}\int_{0}^{t}(t-s)^{\alpha-1}\overline{f}(\overline{X}(s))\diff s,\quad \forall~t\ge 0.
\end{equation}

\begin{propo}
Let $\alpha\in(0,1)$. For any $T\in(0,\infty)$, there exists a unique solution $\bigl(\overline{X}(t)\bigr)_{t\in[0,T]}$ to~\eqref{eq:averaged}.
\end{propo}

\section{Main results}\label{sec:main}

Let us now state the two main results of this article. First, Theorem~\ref{theo:1} justifies the averaging principle, i.e. the convergence of the slow component $X^\epsilon$ of~\eqref{eq:system} to the solution $\overline{X}$ to the averaged equation~\eqref{eq:averaged}, in the setting described in Section~\ref{sec:setting}.

\begin{theo}\label{theo:1}
For all $T\in(0,\infty)$, $x_0\in\R^p$ and $y_0\in\R^q$, one has
\begin{equation}
\underset{\epsilon\to 0}\lim~\underset{t\in[0,T]}\sup~\E[\|{X}^{\epsilon}(t)-\overline{X}(t)\|^2]=0.
\end{equation}
\end{theo}

Theorem~\ref{theo:1} does not provide a rate of convergence with respect to $\epsilon$ in the averaging principle. Theorem~\ref{theo:2} provides a result in that direction, when it is assumed that the fast component $Y^\epsilon$ does not depend on the slow component $X^\epsilon$, i.e. that the mapping $b$ and $\sigma$ satisfy $b(x,y)=b(0,y)$ and $\sigma(x,y)=\sigma(0,y)$. One obtains the rate of convergence $\alpha/2$.

\begin{theo}\label{theo:2}
Assume that the mappings $b$ and $\sigma$ are independent of the slow component, i.e. that $b(x,y)=b(0,y)$ and $\sigma(x,y)=\sigma(0,y)$ for all $x\in\R^p$ and $y\in\R^q$.

For all $T\in(0,\infty)$, there exists $C_\alpha(T)\in(0,\infty)$ such that for all $x_0\in\R^p$ and $y_0\in\R^q$, one has
\begin{equation}\label{eq:theo2}
\underset{t\in[0,T]}\sup~\bigl(\E[\|{X}^{\epsilon}(t)-\overline{X}(t)\|^2]\bigr)^{\frac12}\le C_\alpha(T)\bigl(1+\|x_0\|+\|y_0\|\bigr)\epsilon^{\frac{\alpha}{2}}.
\end{equation}
\end{theo}

\begin{rem}
For any $\epsilon\in(0,1)$, introduce the processes $\bigl(\mathcal{X}^\epsilon(t)\bigr)_{t\ge 0}$ and $\bigl(\mathcal{Y}^\epsilon(t)\bigr)_{t\ge 0}$ defined by
\[
\mathcal{X}^\epsilon(t)=X^\epsilon(\epsilon t),\quad \mathcal{Y}^\epsilon(t)=Y^\epsilon(\epsilon t),\qquad \forall~t\ge 0.
\]
Those processes are solutions to the system
\[
\left\lbrace
\begin{aligned}
\mathcal{X}^\epsilon(t)&=x_0+\frac{\epsilon^\alpha}{\Gamma(\alpha)}\int_{0}^{t}(t-s)^{\alpha-1}f\left(\mathcal{X}^\epsilon(s),\mathcal{Y}^\epsilon(s)\right) \diff s,\quad \forall~t\ge 0,\\
\diff \mathcal{Y}^\epsilon(t)&=b(\mathcal{X}^\epsilon(t),\mathcal{Y}^\epsilon(t))\diff t+\sigma(\mathcal{X}^\epsilon(t),\mathcal{Y}^\epsilon(t)) \diff \tilde{B}^{\epsilon}(t),\quad \forall~t\ge 0,\\
\mathcal{Y}^\epsilon(0)&=y_0,
\end{aligned}
\right.
\]
which is often considered in the literature. Note that the fast equation is driven by a standard Brownian motion $\bigl(\tilde{B}^{\epsilon}(t)\bigr)_{t\ge 0}$ which depends on $\epsilon$, defined by $\tilde{B}^\epsilon(t)=\epsilon^{-\frac12}B(\epsilon t)$.

Define also the process $\bigl(\overline{\mathcal{X}}^\epsilon(t)\bigr)_{t\ge 0}$ by
\[
\overline{\mathcal{X}}^\epsilon(t)=\overline{X}(\epsilon t),\quad \forall~t\ge 0
\]
which is the solution to
\[
\overline{\mathcal{X}}^\epsilon(t)=x_0+\frac{\epsilon^\alpha}{\Gamma(\alpha)}\int_{0}^{t}(t-s)^{\alpha-1}\overline{f}(\overline{\mathcal{X}}^\epsilon(s))\diff s,\quad \forall~t\ge 0.
\]
The results of Theorems~\ref{theo:1} and~\ref{theo:2} can be written as
\[
\underset{\epsilon\to 0}\lim~\underset{t\in[0,\frac{T}{\epsilon}]}\sup~\E[\|\mathcal{X}^{\epsilon}(t)-\overline{\mathcal{X}}(t)\|^2]=0
\]
and
\[
\underset{t\in[0,\frac{T}{\epsilon}]}\sup~\bigl(\E[\|\mathcal{X}^{\epsilon}(t)-\overline{\mathcal{X}}(t)\|^2]\bigr)^{\frac12}\le C_\alpha(T)\bigl(1+\|x_0\|+\|y_0\|\bigr)\epsilon^{\frac{\alpha}{2}}
\]
respectively. It is necessary to consider times of the order $T/\epsilon$ in order to observe the averaging effect.
\end{rem}

\section{Auxiliary results}\label{sec:aux}

In the proofs below, the value of $C\in(0,\infty)$ (or $C(T)$ or $C_\alpha(T)$) may vary from line to line.

\subsection{Moment bounds}

\begin{propo}\label{propo:boundXepsYeps}
For any time $T\in(0,\infty)$, there exists $C(T)\in(0,\infty)$ such that for any initial values $x_0\in\R^p$ and $y_0\in\R^q$ one has
\begin{equation}\label{eq:boundXepsYeps}
\underset{\epsilon\in(0,1)}\sup~\underset{t\in[0,T]}\sup~\E\bigl[\|X^\epsilon(t)\|^2+\|Y^\epsilon(t)\|^2\bigr]\le C(T)\bigl(1+\|x_0\|^2+\|y_0\|^2\bigr).
\end{equation}
\end{propo}

\begin{proof}[Proof of Proposition~\ref{propo:boundXepsYeps}]
On the one hand, recall that the mapping $f$ has at most polynomial growth, see the inequality~\eqref{eq:lineargrowth}. As a result, for all $t\in[0,T]$ one has
\[
\|X^\epsilon(t)\|\le \|x_0\|+C\int_{0}^{t}(t-s)^{\alpha-1}\bigl(1+\|X^\epsilon(s)\|+\|Y^\epsilon(s)\|\bigr)\diff s.
\]
In addition, for all $\alpha\in(0,1)$ and $T\in(0,\infty)$ one has
\[
\int_{0}^{t}(t-s)^{\alpha-1}\diff s=\frac{t^\alpha}{\alpha}\le \frac{T^\alpha}{\alpha},\quad \forall~t\in[0,T].
\]
Appplying the Cauchy--Schwarz inequality, using the inequality above and taking expectation, one obtains for all $t\in[0,T]$
\begin{equation}\label{eq:ineqXeps}
\E[\|X^\epsilon(t)\|^2]\le 2\|x_0\|^2+C(T)\int_{0}^{t}(t-s)^{\alpha-1}\bigl(1+\E[\|X^\epsilon(s)\|^2]+\E[\|Y^\epsilon(s)\|^2]\bigr)\diff s.
\end{equation}
On the other hand, applying the It\^o formula and using the condition~\eqref{eq:bounddissip} satisfied by $b$ and $\sigma$, one obtains for all $t\in[0,T]$
\begin{align*}
\frac12\frac{\diff \E[\|Y^\epsilon(t)\|^2]}{\diff t}&=\frac{1}{\epsilon}\E[\langle b(X^\epsilon(t),Y^\epsilon(t)),Y^\epsilon(t)\rangle]+\frac{1}{2\epsilon}\E[\|\sigma(X^\epsilon,Y^\epsilon(t))\|^2]\\
&\le -\frac{\gamma}{2\epsilon}\E[\|Y^\epsilon(t)\|^2]+\frac{C}{\epsilon}(1+\E[\|X^\epsilon(t)\|^2).
\end{align*}
Applying the Gr\"onwall inequality then yields the following inequality: for all $t\in[0,T]$ one has
\begin{align}
\E[\|Y^\epsilon(t)\|^2]&\le e^{-\frac{\gamma t}{2\epsilon}}\|y_0\|^2+\frac{C}{\epsilon}\int_0^t e^{-\frac{\gamma (t-s)}{2\epsilon}}(1+\E[\|X^\epsilon(s)\|^2)\diff s \nonumber \\
&\le \|y_0\|^2+\frac{2C}{\gamma}+\frac{C}{\epsilon}\int_0^t e^{-\frac{\gamma (t-s)}{2\epsilon}}\E[\|X^\epsilon(s)\|^2]\diff s.\label{eq:ineqYeps}
\end{align}
Plugging the inequality~\eqref{eq:ineqYeps} in~\eqref{eq:ineqXeps}, one obtains
\begin{align*}
\E[\|X^\epsilon(t)\|^2]&\le C(T)\bigl(1+\|x_0\|^2+\|y_0\|^2\bigr)+C(T)\int_0^t (t-s)^{\alpha-1}\E[\|X^\epsilon(s)\|^2]\diff s\\
&\quad +\frac{C}{\epsilon}\int_{0}^{t}(t-s)^{\alpha-1}\int_{0}^{s}e^{-\frac{\gamma(s-r)}{2\epsilon}}\E[\|X^\epsilon(r)\|^2]\diff r \diff s.
\end{align*}
Applying the Fubini theorem, for all $t\in[0,T]$ one obtains
\begin{align*}
\frac{1}{\epsilon}\int_{0}^{t}(t-s)^{\alpha-1}\int_{0}^{s}e^{-\frac{\gamma(s-r)}{2\epsilon}}\E[\|X^\epsilon(r)\|^2]\diff r \diff s&=\frac{1}{\epsilon}\int_{0}^{t}\int_{r}^{t}(t-s)^{\alpha-1}e^{-\frac{\gamma(s-r)}{2\epsilon}} \diff s \E[\|X^\epsilon(r)\|^2]\diff r.
\end{align*}
In addition, applying the change of variable $s=(1-\theta)r+\theta t$ for $\theta\in[0,1]$ one has
\begin{align*}
\frac{1}{\epsilon}\int_{r}^{t}(t-s)^{\alpha-1}e^{-\frac{\gamma(s-r)}{2\epsilon}} \diff s&=(t-r)^{\alpha-1}\int_{0}^{1}(1-\theta)^{\alpha-1}\frac{t-r}{\epsilon}e^{-\frac{\gamma \theta (t-r)}{2\epsilon}} \diff \theta.
\end{align*}
The integral on the right-hand side above can be bounded as follows: for all $\epsilon\in(0,1)$ one has
\begin{align*}
\int_{0}^{1}(1-\theta)^{\alpha-1}\frac{t-r}{\epsilon}e^{-\frac{\gamma \theta (t-r)}{2\epsilon}} \diff\theta&=\int_{0}^{1/2}(1-\theta)^{\alpha-1}\frac{t-r}{\epsilon}e^{-\frac{\gamma \theta (t-r)}{2\epsilon}} \diff\theta\\
&+\int_{1/2}^{1}(1-\theta)^{\alpha-1}\frac{t-r}{\epsilon}e^{-\frac{\gamma \theta (t-r)}{2\epsilon}} \diff\theta\\
&\le 2^{\alpha-1}\int_{0}^{1/2}\frac{t-r}{\epsilon}e^{-\frac{\gamma \theta (t-r)}{2\epsilon}} \diff\theta\\
&+\frac{2e^{-1}}{\gamma}\int_{1/2}^{1}(1-\theta)^{\alpha-1} \diff\theta\\
&\le C
\end{align*}
where $C\in(0,\infty)$ is independent of $\epsilon$, using the elementary upper bounds
\begin{align*}
\frac{t-r}{\epsilon}e^{-\frac{\gamma \theta (t-r)}{2\epsilon}}\le \frac{\underset{z\ge 0}\sup~ze^{-z}}{\gamma\theta}\le \frac{2e^{-1}}{\gamma},\quad \forall~\theta\in[1/2,1].
\end{align*}
Combining the upper bounds obtained above, for all $t\in[0,T]$ one has
\[
\E[\|X^\epsilon(t)\|^2]\le C(T)\bigl(1+\|x_0\|^2+\|y_0\|^2\bigr)+C(T)\int_0^t (t-s)^{\alpha-1}\E[\|X^\epsilon(s)\|^2]\diff s,
\]
where $C(T)\in(0,\infty)$ does not depend on $\epsilon\in(0,1)$. Applying a version of the Gr\"onwall inequality, one obtains
\begin{equation}\label{eq:boundXeps}
\underset{t\in[0,T]}\sup~\E[\|X^\epsilon(t)\|^2]\le C(T)\bigl(1+\|x_0\|^2+\|y_0\|^2\bigr).
\end{equation}
Plugging the inequality~\eqref{eq:boundXeps} in the auxiliary inequality~\eqref{eq:ineqYeps}, one obtains
\begin{equation}\label{eq:boundYeps}
\underset{t\in[0,T]}\sup~\E[\|Y^\epsilon(t)\|^2]\le C(T)\bigl(1+\|x_0\|^2+\|y_0\|^2\bigr).
\end{equation}
Note that $C(T)\in(0,\infty)$ appearing in the upper bounds~\eqref{eq:boundXeps} and~\eqref{eq:boundYeps} above does not depend on $\epsilon\in (0,1)$. Therefore one obtains the inequality~\eqref{eq:boundXepsYeps} and the proof of Proposition~\ref{propo:boundXepsYeps} is completed.
\end{proof}

\begin{propo}\label{propo:boundXbar}
For any time $T\in(0,\infty)$, there exists $C(T)\in(0,\infty)$ such that for any initial value $x_0\in\R^p$ one has
\begin{equation}\label{eq:boundXbar}
\underset{t\in[0,T]}\sup~\|\overline{X}(t)\| \le C(T)\bigl(1+\|x_0\|\bigr).
\end{equation}
\end{propo}

\begin{proof}[Proof of Proposition~\ref{propo:boundXbar}]
Let $T\in(0,\infty)$ and $x_0\in\R^p$ be given.

Since the mapping $\overline{f}$ is globally Lipschitz continuous, it has at most linear growth. As a result, there exists $C\in(0,\infty)$ such that for all $t\in[0,T]$ one has
\begin{align*}
\|\overline{X}(t)\|&\le \|x_0\|+C\int_{0}^{t}(t-s)^{\alpha-1}\diff s +C\int_{0}^{t}(t-s)^{\alpha-1} \|\overline{X}(s)\| \diff s\\
&\le C(\|x_0\|+T^\alpha)+C\int_{0}^{t}(t-s)^{\alpha-1} \|\overline{X}(s)\| \diff s.
\end{align*}
Applying a version of the Gr\"onwall lemma yields the inequality~\eqref{eq:boundXbar} and concludes the proof of Proposition~\ref{propo:boundXbar}.
\end{proof}

\subsection{Regularity properties}

Below, the following elementary inequality is employed.
\begin{lemma}\label{lem:ineq}
For all $\beta\in[0,\alpha]$, there exists $C_\beta\in(0,\infty)$ such that one has
\begin{equation}\label{eq:ineq}
|r_2^{\alpha-1}-r_1^{\alpha-1}|\le C_\beta(T)|r_2-r_1|^{\beta}\min(r_1,r_2)^{\alpha-\beta-1},\quad \forall~r_1,r_2\in(0,\infty).
\end{equation}
\end{lemma}

\begin{proof}[Proof of Lemma~\ref{lem:ineq}]
On the one hand, applying the fundamental theorem of calculus, one obtains for all $r_1,r_2\in(0,\infty)$
\[
|r_2^{\alpha-1}-r_1^{\alpha-1}|\le (1-\alpha) |r_2-r_1|\min(r_1,r_2)^{\alpha-2}.
\]
On the other hand, applying the triangle inequality, one obtains for all $r_1,r_2\in(0,\infty)$
\[
|r_2^{\alpha-1}-r_1^{\alpha-1}|\le |r_2^{\alpha-1}| + |r_1^{\alpha-1}|\le 2\min(r_1,r_2)^{\alpha-1}.
\]
Given $\beta\in[0,\alpha]$, an interpolation of the two inequalities above gives for all $r_1,r_2\in(0,\infty)$
\begin{align*}
|r_2^{\alpha-1}-r_1^{\alpha-1}|&\le |r_2^{\alpha-1}-r_1^{\alpha-1}|^{\beta} |r_2^{\alpha-1}-r_1^{\alpha-1}|^{1-\beta}\\
&\le (1-\alpha)^{\beta}2^{1-\beta} |r_2-r_1|^\beta \min(r_1,r_2)^{\beta(\alpha-2)+(1-\beta)(\alpha-1)},
\end{align*}
with $\beta(\alpha-2)+(1-\beta)(\alpha-1)=\alpha-\beta-1$. The proof of Lemma~\ref{lem:ineq} is thus completed.
\end{proof}

\begin{propo}\label{propo:regulXeps}
For all $\beta\in(0,\alpha)$ and any time $T\in(0,\infty)$, there exists $C_\beta(T)\in(0,\infty)$ such that, any initial values $x_0\in\R^p$ and $y_0\in\R^q$, one has
\begin{equation}\label{eq:regulXeps}
\bigl(\E[\|X^\epsilon(t_2)-X^\epsilon(t_1)\|^2]\bigr)^{\frac12}\le C_\beta(T)\bigl(1+\|x_0\|+\|y_0\|\bigr)|t_2-t_1|^\beta,\quad \forall~t_1,t_2\in[0,T].
\end{equation}
\end{propo}

\begin{proof}[Proof of Proposition~\ref{propo:regulXeps}]
Without loss of generality, assume that $0\le t_1<t_2\le T$, then one can decompose $X^\epsilon(t_2)-X^\epsilon(t_1)$ as follows: one has
\begin{align*}
X^\epsilon(t_2)-X^\epsilon(t_1)&=\frac{1}{\Gamma(\alpha)}\int_{0}^{t_2}(t_2-s)^{\alpha-1}f(X^\epsilon(s),Y^\epsilon(s))\diff s-\frac{1}{\Gamma(\alpha)}\int_{0}^{t_1}(t_1-s)^{\alpha-1}f(X^\epsilon(s),Y^\epsilon(s))\diff s\\
&=\frac{1}{\Gamma(\alpha)}\int_{0}^{t_1}\left[(t_2-s)^{\alpha-1}-(t_1-s)^{\alpha-1}\right]f(X^\epsilon(s),Y^\epsilon(s))\diff s\\
&+\frac{1}{\Gamma(\alpha)}\int_{t_1}^{t_2}(t_2-s)^{\alpha-1}f(X^\epsilon(s),Y^\epsilon(s))\diff s.
\end{align*}
Recall that the mapping $f$ has at most linear growth, owing to the inequality~\eqref{eq:lineargrowth}, thus applying the inequality~\eqref{eq:boundXepsYeps} from Proposition~\ref{propo:boundXepsYeps}, one obtains
\begin{align*}
\underset{s\in[0,T]}\sup~\bigl(\E[\|f(X^\epsilon(s),Y^\epsilon(s))\|^2]\bigr)^{\frac12}&\le C\Bigl(1+\underset{s\in[0,T]}\sup~\bigl(\E[\|X^\epsilon(s)\|^2]\bigr)^{\frac12}+\underset{s\in[0,T]}\sup~\bigl(\E[\|Y^\epsilon(s)\|^2]\bigr)^{\frac12}\Bigr)\\
&\le C\bigl(1+\|x_0\|+\|y_0\|\bigr),
\end{align*}
where $C=C(T)$ only depends on $T$, and is independent of $\epsilon$. Applying the Minkowski inequality, one thus obtains
\begin{align*}
\bigl(\E[\|X^\epsilon(t_2)-X^\epsilon(t_1)\|^2]\bigr)^{\frac12}&\le C\bigl(1+\|x_0\|+\|y_0\|\bigr)\int_{0}^{t_1}\left|(t_2-s)^{\alpha-1}-(t_1-s)^{\alpha-1}\right|\diff s\\
&+C\bigl(1+\|x_0\|+\|y_0\|\bigr)\int_{t_1}^{t_2}(t_2-s)^{\alpha-1}\diff s.
\end{align*}
For the first term, applying the inequality~\eqref{eq:ineq} from Lemma~\ref{lem:ineq} (with $\beta\in(0,\alpha)$, $r_2=t_2-s$ and $r_1=t_1-s$), one obtains for all $t_1,t_2\in[0,T]$
\begin{align*}
\int_{0}^{t_1}\left|(t_2-s)^{\alpha-1}-(t_1-s)^{\alpha-1}\right|\diff s &\le C_\beta|t_2-t_1|^\beta \int_{0}^{t_1}(t_1-s)^{\alpha-\beta-1}\diff s\\
&\le \frac{C_\beta T^{\alpha-\beta}}{\alpha-\beta}|t_2-t_1|^\beta. 
\end{align*}
For the second term, one obtains for all $t_1,t_2\in[0,T]$
\[
\int_{t_1}^{t_2}(t_2-s)^{\alpha-1}\diff s=\frac{|t_2-t_1|^\alpha}{\alpha}\le \frac{T^{\alpha-\beta}}{\alpha}|t_2-t_1|^\beta.
\]
Gathering the estimates yields the inequality~\eqref{eq:regulXeps} and concludes the proof of Proposition~\ref{propo:regulXeps}.
\end{proof}

\begin{propo}\label{propo:regulXbar}
For all $\beta\in(0,\alpha)$ and any time $T\in(0,\infty)$, there exists $C_\beta(T)\in(0,\infty)$ such that for any initial value $x_0\in\R^p$, one has
\begin{equation}\label{eq:regulXbar}
\|\overline{X}(t_2)-\overline{X}(t_1)\|\le C_\beta(T)\bigl(1+\|x_0\|\bigr)|t_2-t_1|^\beta,\quad \forall~t_1,t_2\in[0,T].
\end{equation}
\end{propo}

\begin{proof}[Proof of Proposition~\ref{propo:regulXbar}]
Without loss of generality, assume that $0\le t_1<t_2\le T$, then one can decompose $\overline{X}(t_2)-\overline{X}(t_1)$ as follows: one has
\begin{align*}
\overline{X}(t_2)-\overline{X}(t_1)&=\frac{1}{\Gamma(\alpha)}\int_{0}^{t_2}(t_2-s)^{\alpha-1}\overline{f}(\overline{X}(s))\diff s-\frac{1}{\Gamma(\alpha)}\int_{0}^{t_1}(t_1-s)^{\alpha-1}\overline{f}(\overline{X}(s))\diff s\\
&=\frac{1}{\Gamma(\alpha)}\int_{0}^{t_1}\left[(t_2-s)^{\alpha-1}-(t_1-s)^{\alpha-1}\right]\overline{f}(\overline{X}(s))\diff s\\
&+\frac{1}{\Gamma(\alpha)}\int_{t_1}^{t_2}(t_2-s)^{\alpha-1}\overline{f}(\overline{X}(s))\diff s.
\end{align*}
Recall that the mapping $\overline{f}$ is globally Lipschitz continuous and thus has at most linear growth, thus applying the inequality~\eqref{eq:boundXbar} from Proposition~\ref{propo:boundXbar}, one obtains
\[
\underset{t\in[0,T]}\sup~\|\overline{f}(\overline{X}(t))\|\le C\bigl(1+\underset{t\in[0,T]}\sup~\|\overline{X}(t)\|\bigr)\le C\bigl(1+\|x_0\|\bigr),
\]
where $C=C(T)$ only depends on $T$. Applying the Minkowski inequality, one thus obtains
\begin{align*}
\|\overline{X}(t_2)-\overline{X}(t_1)\|&\le C\bigl(1+\|x_0\|\bigr)\int_{0}^{t_1}\left|(t_2-s)^{\alpha-1}-(t_1-s)^{\alpha-1}\right|\diff s\\
&+C\bigl(1+\|x_0\|\bigr)\int_{t_1}^{t_2}(t_2-s)^{\alpha-1}\diff s.
\end{align*}
Repeating the arguments from the proof of Proposition~\ref{propo:regulXeps} yields the inequality~\eqref{eq:regulXbar} and concludes the proof of Proposition~\ref{propo:regulXbar}.
\end{proof}

\section{Proof of the main results}\label{sec:proofs}

\subsection{Auxiliary system}

Let $\delta\in(0,1)$ be an auxiliary parameter. Without loss of generality, it is assumed that $\delta=T/N$ for some integer $N\in\N$. For all $n\in\{0,\ldots,N\}$, let $t_n=n\delta$.

For all $t\in[0,T]$, set
\[
n_\delta(t)=\lfloor \frac{t}{\delta}\rfloor,
\]
and note that $n_{\delta}(t)=n$ if and only if $t_n\le t<t_{n+1}$.

For any auxiliary parameter $\delta\in(0,1)$ and any $\epsilon\in(0,1)$, introduce the auxiliary process $\bigl(\widehat{X}^{\epsilon,\delta}(t),\widehat{Y}^{\epsilon,\delta}(t)\bigr)_{t\in[0,T]}$ defined as the solution to the system
\begin{equation}\label{eq:aux}
\left\lbrace
\begin{aligned}
\widehat{X}^{\epsilon,\delta}(t)&=x_0+\frac{1}{\Gamma(\alpha)}\int_{0}^{t}(t-s)^{\alpha-1}f\left(X^{\epsilon}(t_{n_\delta(s)}),\widehat{Y}^{\epsilon,\delta}(s)\right) \diff s, \quad \forall~t\ge 0,\\
\diff \widehat{Y}^{\epsilon,\delta}(t)&=\frac{1}{\epsilon}b(X^{\epsilon}(t_{n_\delta(t)}),\widehat{Y}^{\epsilon,\delta}(t))\diff t+\frac{1}{\sqrt{\epsilon}}\sigma(X^{\epsilon}(t_{n_\delta(t)}),\widehat{Y}^{\epsilon,\delta}(t))\diff B(t), \quad \forall~t\ge 0,\\
\widehat{Y}^{\epsilon,\delta}(0)&=y_0,
\end{aligned}
\right.
\end{equation}
where $\bigl(X^\epsilon(t)\bigr)_{t\in[0,T]}$ is given by solving the multiscale stochastic system~\eqref{eq:system}.

More precisely, the auxiliary process $\bigl(\widehat{X}^{\epsilon,\delta}(t),\widehat{Y}^{\epsilon,\delta}(t)\bigr)_{t\in[0,T]}$ is continuous on the interval $[0,T]$, and for any $n\in\{0,\ldots,N-1\}$, on the interval $[t_n,t_{n+1}]$, it is defined as follows. First, the process $\bigl(\widehat{Y}^{\epsilon,\delta}(t)\bigr)_{t\in[t_n,t_{n+1}]}$ is solution to the stochastic differential equation
\[
\diff \widehat{Y}^{\epsilon,\delta}(t)=\frac{1}{\epsilon}b(X^{\epsilon}(t_{n}),\widehat{Y}^{\epsilon,\delta}(t))\diff t+\frac{1}{\sqrt{\epsilon}}\sigma(X^{\epsilon}(t_{n}),\widehat{Y}^{\epsilon,\delta}(t))\diff B(t), \quad \forall~t\in [t_n,t_{n+1}].
\]
Second, the process $\bigl(\widehat{X}^{\epsilon,\delta}(t)\bigr)_{t\in[t_n,t_{n+1}]}$ is given by the following expression: for all $t\in[t_n,t_{n+1}]$ one has
\begin{align*}
\widehat{X}^{\epsilon,\delta}(t)&=x_0+\frac{1}{\Gamma(\alpha)}\sum_{k=0}^{n-1}\int_{t_{k}}^{t_{k+1}}(t-s)^{\alpha-1}f\left(X^{\epsilon}(t_{k}),\widehat{Y}^{\epsilon,\delta}(s)\right) \diff s\\
&+\int_{t_n}^{t}(t-s)^{\alpha-1}f\left(X^{\epsilon}(t_{n}),\widehat{Y}^{\epsilon,\delta}(s)\right) \diff s.
\end{align*}
As a result, the auxiliary system~\eqref{eq:aux} admits a unique solution for any $\delta,\epsilon\in(0,1)$.

Decomposing the error as
\begin{equation}\label{eq:decomperror}
X^\epsilon(t)-\overline{X}(t)=X^\epsilon(t)-\widehat{X}^{\epsilon,\delta}(t)+\widehat{X}^{\epsilon,\delta}(t)-\overline{X}(t),\quad \forall~t\in[0,T],
\end{equation}
Theorem~\ref{theo:1} is a straightforward consequence of Lemma~\ref{lem:errorX1} and Lemma~\ref{lem:errorX2} stated below.

\subsection{Auxiliary error bounds}

\begin{lemma}\label{lem:errorX1}
For all $\beta\in(0,\alpha)$ and any time $T\in(0,\infty)$, there exists $C_\beta(T)\in(0,\infty)$ such that, for any initial values $x_0\in\R^p$ and $y_0\in\R^q$, one has for all $\delta\in(0,1)$
\begin{equation}\label{eq:errorX1}
\underset{\epsilon\in(0,1)}\sup~\underset{t\in[0,T]}\sup~\E[\|X^\epsilon(t)-\widehat{X}^{\epsilon,\delta}(t)\|^2]\le C_\beta(T)(1+\|x_0\|+\|y_0\|)^2\delta^{2\beta}.
\end{equation}
\end{lemma}

\begin{lemma}\label{lem:errorX2}
For all $\beta\in(0,\alpha)$ and any time $T\in(0,\infty)$, there exists $C_\beta(T)\in(0,\infty)$ such that, for any initial values $x_0\in\R^p$ and $y_0\in\R^q$, one has for all $\delta\in(0,1)$
\begin{equation}\label{eq:errorX2}
\underset{t\in[0,T]}\sup~\underset{\epsilon\to 0}\limsup~\E[\|\widehat{X}^{\epsilon,\delta}(t)-\overline{X}(t)\|^2]\le C_\beta(T)(1+\|x_0\|+\|y_0\|)^2\delta^{2\beta}.
\end{equation}
\end{lemma}

The proof of Lemma~\ref{lem:errorX1} is based on the result stated in  Lemma~\ref{lem:errorY}.
\begin{lemma}\label{lem:errorY}
For all $\beta\in(0,\alpha)$ and any time $T\in(0,\infty)$, there exists $C_\beta(T)\in(0,\infty)$ such that, for any initial values $x_0\in\R^p$ and $y_0\in\R^q$, one has for all $\epsilon\in(0,1)$ and $\delta\in(0,1)$
\begin{equation}\label{eq:errorY}
\underset{t\in[0,T]}\sup~\E[\|Y^\epsilon(t)-\widehat{Y}^{\epsilon,\delta}(t)\|^2]\le C_\beta(T)(1+\|x_0\|+\|y_0\|)^2\delta^{2\beta}.
\end{equation}
\end{lemma}

\begin{rem}\label{rem:boundYaux}
Note that combining the moment bounds~\eqref{eq:boundXepsYeps} on $Y^\epsilon$ from Proposition~\ref{propo:boundXepsYeps} and the error bound~\eqref{eq:errorY} from Lemma~\ref{lem:errorY}, one obtains moment bounds for the auxiliary process $\widehat{Y}^{\epsilon,\delta}$, uniformly with respect to the parameters $\epsilon,\delta\in(0,1)$: for all $T\in(0,\infty)$ there exists $C(T)\in(0,\infty)$ such that, for any initial values $x_0\in\R^p$ and $y_0\in\R^q$, one has for all $\epsilon\in(0,1)$ and $\delta\in(0,1)$
\begin{equation}\label{eq:boundYaux}
\underset{t\in[0,T]}\sup~\E[\|\widehat{Y}^{\epsilon,\delta}(t)\|^2]\le C(T)\bigl(1+\|x_0\|+\|y_0\|\bigr)^2.
\end{equation}
\end{rem}

\begin{proof}[Proof of Lemma~\ref{lem:errorY}]
Owing to~\eqref{eq:system} and to~\eqref{eq:aux}, for all $t\in[0,T]$ one has
\begin{align*}
\diff\bigl(Y^\epsilon(t)-\widehat{Y}^{\epsilon,\delta}(t)\bigr)&=\frac{1}{\epsilon}\Bigl(b(X^\epsilon(t),Y^\epsilon(t))-b(X^\epsilon(t_{n_\delta(t)}),\widehat{Y}^{\epsilon,\delta}(t))\Bigr)\diff t\\
&~+\frac{1}{\sqrt{\epsilon}}\Bigl(\sigma(X^\epsilon(t),Y^\epsilon(t))-\sigma(X^\epsilon(t_{n_\delta(t)}),\widehat{Y}^{\epsilon,\delta}(t))\Bigr)\diff B(t)\\
&=\frac{1}{\epsilon}\Bigl(b(X^\epsilon(t),Y^\epsilon(t))-b(X^\epsilon(t),\widehat{Y}^{\epsilon,\delta}(t))\Bigr)\diff t\\
&~+\frac{1}{\sqrt{\epsilon}}\Bigl(\sigma(X^\epsilon(t),Y^\epsilon(t))-\sigma(X^\epsilon(t),\widehat{Y}^{\epsilon,\delta}(t))\Bigr)\diff B(t)\\
&~+\frac{1}{\epsilon}\Bigl(b(X^\epsilon(t),\widehat{Y}^{\epsilon,\delta}(t))-b(X^\epsilon(t_{n_\delta(t)}),\widehat{Y}^{\epsilon,\delta}(t))\Bigr)\diff t\\
&~+\frac{1}{\sqrt{\epsilon}}\Bigl(\sigma(X^\epsilon(t),\widehat{Y}^{\epsilon,\delta}(t))-\sigma(X^\epsilon(t_{n_\delta(t)}),\widehat{Y}^{\epsilon,\delta}(t))\Bigr)\diff B(t).
\end{align*}
Applying It\^o's formula and using the Lipschitz continuity properties of the mappings $b$ and $\sigma$ from Assumption~\ref{ass:b-sigma}, one obtains for all $t\in[0,T]$
\begin{align*}
\frac12\frac{\diff \E[\|Y^\epsilon(t)-\widehat{Y}^{\epsilon,\delta}(t)\|^2]}{\diff t}
&=\frac{1}{\epsilon}\E\bigl[\langle b(X^\epsilon(t),Y^\epsilon(t))-b(X^\epsilon(t),\widehat{Y}^{\epsilon,\delta}(t)),Y^\epsilon(t)-\widehat{Y}^{\epsilon,\delta}(t)\rangle\bigr] \\
&~+\frac{1}{2\epsilon}\E\bigl[\big\|\sigma(X^\epsilon(t),Y^\epsilon(t))-\sigma(X^\epsilon(t),\widehat{Y}^{\epsilon,\delta}(t))\big\|^2\bigr]\\
&~+\frac{1}{\epsilon}\E\bigl[\langle b(X^\epsilon(t),\widehat{Y}^{\epsilon,\delta}(t))-b(X^\epsilon(t_{n_\delta(t)}),\widehat{Y}^{\epsilon,\delta}(t)),Y^\epsilon(t)-\widehat{Y}^{\epsilon,\delta}(t)\rangle \bigr]\\
&~+\frac{1}{2\epsilon}\E\bigl[\big\|\sigma(X^\epsilon(t),\widehat{Y}^{\epsilon,\delta}(t))-\sigma(X^\epsilon(t_{n_\delta(t)}),\widehat{Y}^{\epsilon,\delta}(t))\big\|^2\bigr]\\
&\le \frac{1}{\epsilon}\E\bigl[\langle b(X^\epsilon(t),Y^\epsilon(t))-b(X^\epsilon(t),\widehat{Y}^{\epsilon,\delta}(t)),Y^\epsilon(t)-\widehat{Y}^{\epsilon,\delta}(t)\rangle\bigr] \\
&~+\frac{1}{2\epsilon}\E\bigl[\big\|\sigma(X^\epsilon(t),Y^\epsilon(t))-\sigma(X^\epsilon(t),\widehat{Y}^{\epsilon,\delta}(t))\big\|^2\bigr]\\
&~+\frac{C}{\epsilon}\E[\|X^\epsilon(t)-X^\epsilon(t_{n_\delta(t)})\|\|Y^\epsilon(t)-\widehat{Y}^{\epsilon,\delta}(t)\|]\\
&~+\frac{C}{\epsilon}\E[\|X^\epsilon(t)-X^\epsilon(t_{n_\delta(t)})\|^2].
\end{align*}
Then, owing to the condition~\eqref{eq:dissipation} from Assumption~\ref{ass:b-sigma} and using the Cauchy--Schwarz and Young inequalities, one obtains for all $t\in[0,T]$
\begin{align*}
\frac12&\frac{\diff \E[\|Y^\epsilon(t)-\widehat{Y}^{\epsilon,\delta}(t)\|^2]}{\diff t}\le -\frac{\gamma}{\epsilon}\E[\|Y^\epsilon(t)-\widehat{Y}^{\epsilon,\delta}(t)\|^2]\\
&~+\frac{C}{\epsilon}\E[\|X^\epsilon(t)-X^\epsilon(t_{n_\delta(t)})\|\|Y^\epsilon(t)-\widehat{Y}^{\epsilon,\delta}(t)\|]+\frac{C}{\epsilon}\E[\|X^\epsilon(t)-X^\epsilon(t_{n_\delta(t)})\|^2]\\
&\le -\frac{\gamma}{2\epsilon}\E[\|Y^\epsilon(t)-\widehat{Y}^{\epsilon,\delta}(t)\|^2]+\frac{C}{\epsilon}\E[\|X^\epsilon(t)-X^\epsilon(t_{n_\delta(t)})\|^2].
\end{align*}
In addition, owing to the inequality~\eqref{eq:regulXeps} from Proposition~\ref{propo:regulXeps}, one obtains
\[
\frac12\frac{\diff \E[\|Y^\epsilon(t)-\widehat{Y}^{\epsilon,\delta}(t)\|^2]}{\diff t}\le -\frac{\gamma}{2\epsilon}\E[\|Y^\epsilon(t)-\widehat{Y}^{\epsilon,\delta}(t)\|^2]+\frac{C_\beta(T,x_0,y_0)\delta^{2\beta}}{\epsilon},
\]
with the notation $C_\beta(T,x_0,y_0)=C_\beta(T)(1+\|x_0\|+\|y_0\|)^2$ used to simplify the presentation here and below.

Then, applying the Gr\"onwall inequality one obtains, for all $n\in\{0,\ldots,N-1\}$ and all $t\in[t_n,t_{n+1}]$
\begin{equation}\label{eq:auxineq}
\begin{aligned}
\E[\|Y^\epsilon(t)-\widehat{Y}^{\epsilon,\delta}(t)\|^2]&\le e^{-\frac{\gamma(t-t_{n})}{\epsilon}}\E[\|Y^\epsilon(t_n)-\widehat{Y}^{\epsilon,\delta}(t_n)\|^2]\\
&~+\frac{C_\beta(T,x_0,y_0)\delta^{2\beta}}{\epsilon}\int_{t_n}^{t}e^{-\frac{\gamma(t-s)}{\epsilon}}\diff s.
\end{aligned}
\end{equation}
For all $n\in\{0,\ldots,N-1\}$, let $\varrho_n^{\epsilon,\delta}=\E[\|Y^\epsilon(t_n)-\widehat{Y}^{\epsilon,\delta}(t_n)\|^2]$. Letting $t=t_{n+1}$ in the inequality above, one obtains for all $n\in\{0,\ldots,N-1\}$
\begin{align*}
\varrho_{n+1}^{\epsilon,\delta}&\le e^{-\frac{\gamma\delta}{\epsilon}}\varrho_n^{\epsilon,\delta}+\frac{C_\beta(T,x_0,y_0)\delta^{2\beta}}{\epsilon}\int_{t_n}^{t_{n+1}}e^{-\frac{\gamma(t_{n+1}-s)}{\epsilon}}\diff s.
\end{align*}
Note that one has $\varrho_0^{\epsilon,\delta}=0$, therefore a discrete Gr\"onwall inequality argument yields for all $n\in\{0,\ldots,N\}$
\begin{align*}
\E[\|Y^\epsilon(t_n)-\widehat{Y}^{\epsilon,\delta}(t_n)\|^2]=\varrho_n^{\epsilon,\delta}&\le \frac{C_\beta(T,x_0,y_0)\delta^{2\beta}}{\epsilon}\sum_{k=0}^{n-1}e^{-\frac{\gamma(t_n-t_{k+1})}{\epsilon}}\int_{t_k}^{t_{k+1}}e^{-\frac{\gamma(t_{k+1}-s)}{\epsilon}}\diff s\\
&\le \frac{C_\beta(T,x_0,y_0)\delta^{2\beta}}{\epsilon}\sum_{k=0}^{n-1}\int_{t_k}^{t_{k+1}}e^{-\frac{\gamma(t_{n}-s)}{\epsilon}}\diff s\\
&\le \frac{C_\beta(T,x_0,y_0)\delta^{2\beta}}{\epsilon}\int_{0}^{t_{n}}e^{-\frac{\gamma(t_{n}-s)}{\epsilon}}\diff s\\
&\le C_\beta(T,x_0,y_0)\delta^{2\beta}.
\end{align*}
Plugging that upper bound in the inequality~\eqref{eq:auxineq}, for all $t\in[0,T]$, one obtains
\begin{align*}
\E[\|Y^\epsilon(t)-\widehat{Y}^{\epsilon,\delta}(t)\|^2]
&\le e^{-\frac{\gamma(t-t_{n_\delta(t)})}{\epsilon}}\E[\|Y^\epsilon(t_{n_\delta(t)})-\widehat{Y}^{\epsilon,\delta}(t_{n_\delta(t)})\|^2]\\
&~+\frac{C_\beta(T,x_0,y_0)\delta^{2\beta}}{\epsilon}\int_{t_{n_\delta (t)}}^{t}e^{-\frac{\gamma(t-s)}{\epsilon}}\diff s\\
&\le C_\beta(T,x_0,y_0)\delta^{2\beta}\Bigl(e^{-\frac{\gamma(t-t_{n_\delta(t)})}{\epsilon}}+\bigl(1-e^{-\frac{\gamma(t-t_{n_\delta(t)})}{\epsilon}}\bigr)\Bigr)\\
&\le C_\beta(T,x_0,y_0)\delta^{2\beta}.
\end{align*} 
Observe that the upper bound above holds for arbitrary $t\in[0,T]$, and recall that one has $C_\beta(T,x_0,y_0)=C_\beta(T)(1+\|x_0\|+\|y_0\|)^2$. This yields the inequality~\eqref{eq:errorY} and the proof of Lemma~\ref{lem:errorY} is completed.
\end{proof}

It remains to prove Lemma~\ref{lem:errorX1} and Lemma~\ref{lem:errorX2}.

\begin{proof}[Proof of Lemma~\ref{lem:errorX1}]
Owing to~\eqref{eq:system} and~\eqref{eq:aux}, for all $t\in[0,T]$ one has
\begin{align*}
X^\epsilon(t)-\widehat{X}^{\epsilon,\delta}(t)&=\frac{1}{\Gamma(\alpha)}\int_{0}^{t}(t-s)^{\alpha-1}\Bigl[f\left(X^\epsilon(s),Y^\epsilon(s)\right)-f\left(X^{\epsilon}(t_{n_\delta(s)}),\widehat{Y}^{\epsilon,\delta}(s)\right)\Bigr] \diff s\\
&=\frac{1}{\Gamma(\alpha)}\int_{0}^{t}(t-s)^{\alpha-1}\Bigl[f\left(X^\epsilon(s),Y^\epsilon(s)\right)-f\left(X^{\epsilon}(t_{n_\delta(s)}),Y^\epsilon(s)\right)\Bigr] ds\\
&+\frac{1}{\Gamma(\alpha)}\int_{0}^{t}(t-s)^{\alpha-1}\Bigl[f\left(X^{\epsilon}(t_{n_\delta(s)}),Y^\epsilon(s)\right)-f\left(X^{\epsilon}(t_{n_\delta(s)}),\widehat{Y}^{\epsilon,\delta}(s)\right)\Bigr] \diff s.
\end{align*}
Since the mapping $f$ is globally Lipschitz continuous (see Assumption~\ref{ass:F}), applying the Minkowski inequality one obtains for all $t\in[0,T]$
\begin{align*}
\bigl(\E[\|X^\epsilon(t)-\hat{X}^{\epsilon,\delta}(t)\|^2]\bigr)^{\frac12}&\le C\int_{0}^{t}(t-s)^{\alpha-1}\bigl(\E[\|X^\epsilon(s)-X^{\epsilon}(t_{n_\delta(s)})\|^2]\bigr)^{\frac12} \diff s\\
&+C\int_{0}^{t}(t-s)^{\alpha-1}\bigl(\E[\|Y^\epsilon(s)-\widehat{Y}^{\epsilon,\delta}(s)\|^2]\bigr)^{\frac12}\diff s.
\end{align*}
Applying the inequality~\eqref{eq:regulXeps} from Proposition~\ref{propo:regulXeps} and the inequality~\eqref{eq:errorY} from Lemma~\ref{lem:errorY}, and noting that
\[
\underset{t\in[0,T]}\sup~\int_{0}^{t}(t-s)^{\alpha-1}\diff s=\frac{T^\alpha}{\alpha},
\]
one obtains
\[
\bigl(\E[\|X^\epsilon(t)-\hat{X}^{\epsilon,\delta}(t)\|^2]\bigr)^{\frac12}\le C_\beta(T)\bigl(1+\|x_0\|+\|y_0\|\bigr)\delta^{\beta}.
\]
This yields the inequality~\eqref{eq:errorX1} and concludes the proof of Lemma~\ref{lem:errorX1}.
\end{proof}

\begin{proof}[Proof of Lemma~\ref{lem:errorX2}]
Owing to~\eqref{eq:aux} and to~\eqref{eq:averaged}, for all $t\in[0,T]$, one has
\begin{align}
\widehat{X}^{\epsilon,\delta}(t)-\overline{X}(t)&=\frac{1}{\Gamma(\alpha)}\int_{0}^{t}(t-s)^{\alpha-1}\Bigl[f\bigl({X}^{\epsilon}(t_{n_\delta(s)}),\widehat{Y}^{\epsilon,\delta}(s)\bigr)-\overline{f}(\overline{X}(s))\Bigr] \diff s \nonumber\\
&=r_1^{\epsilon,\delta}(t)+r_2^{\epsilon,\delta}(t)+r_3^{\epsilon,\delta}(t)+r_4^{\epsilon,\delta}(t),\label{eq:decomperrorX2}
\end{align}
with error terms defined by
\begin{align*}
r_1^{\epsilon,\delta}(t)&=\frac{1}{\Gamma(\alpha)}\int_{0}^{t}(t-s)^{\alpha-1}\bigl[f\bigl({X}^{\epsilon}(t_{n_{\delta}(s)}),\widehat{Y}^{\epsilon,\delta}(s)\bigr)-\overline{f}({X}^{\epsilon}(t_{n_\delta(s)}))\bigr] \diff s\\
r_2^{\epsilon,\delta}(t)&=\frac{1}{\Gamma(\alpha)}\int_{0}^{t}(t-s)^{\alpha-1}\bigl[\overline{f}({X}^{\epsilon}(t_{n_\delta(s)}))-\overline{f}(\widehat{X}^{\epsilon,\delta}(t_{n_\delta(s)}))\bigr] \diff s\\
r_3^{\epsilon,\delta}(t)&=\frac{1}{\Gamma(\alpha)}\int_{0}^{t}(t-s)^{\alpha-1}\bigl[\overline{f}(\widehat{X}^{\epsilon,\delta}(t_{n_\delta(s)}))-\overline{f}(\overline{X}(t_{n_\delta(s)}))\bigr] \diff s\\
r_4^{\epsilon,\delta}(t)&=\frac{1}{\Gamma(\alpha)}\int_{0}^{t}(t-s)^{\alpha-1}\bigl[\overline{f}(\overline{X}(t_{n_\delta(s)}))-\overline{f}(\overline{X}(s))\bigr] \diff s.
\end{align*}

To simplify the presentation, the notation $C_\beta(T,x_0,y_0)=C_\beta(T)\bigl(1+\|x_0\|+\|y_0\|)^2$ is employed below.

$\bullet$ Treatment of the error term $r_1^{\epsilon,\delta}(t)$.

Define the auxiliary mapping $\Delta f$ as follows: set
\begin{equation}\label{eq:Deltaf}
\Delta f(x,y)=f(x,y)-\overline{f}(x),\quad \forall~x\in\R^p,y\in\R^q.
\end{equation}

For all $t\in[0,T]$, one has the decomposition
\begin{align*}
r_1^{\epsilon,\delta}(t)&=\frac{1}{\Gamma(\alpha)}\int_{t_{n_\delta(t)}}^{t}(t-s)^{\alpha-1}\Delta f\bigl({X}^{\epsilon}(t_{n_{\delta}(t)}),\widehat{Y}^{\epsilon,\delta}(s)\bigr) \diff s\\
&~+\sum_{k=0}^{n_\delta(t)-1}\frac{1}{\Gamma(\alpha)}\int_{t_k}^{t_{k+1}}(t-s)^{\alpha-1}\Delta f\bigl({X}^{\epsilon}(t_k),\widehat{Y}^{\epsilon,\delta}(s)\bigr)\diff s.
\end{align*}

Dealing with the first part of the error term $r_1^{\epsilon,\delta}(t)$ is straightforward. Applying the Minkowski inequality, and recalling that the mappings $f$ and $\overline{f}$ have at most linear growth owing to~\eqref{eq:lineargrowth} and~\eqref{eq:lineargrowth2}, one obtains
\begin{align*}
\bigl(\E[&\|\int_{t_{n_\delta(s)}}^{t}(t-s)^{\alpha-1}\Delta f\bigl({X}^{\epsilon}(t_{n_{\delta}(t)}),\widehat{Y}^{\epsilon,\delta}(s)\bigr) \diff s\|^2]\bigr)^{\frac12}\\
&\le \int_{t_{n_\delta(s)}}^{t}(t-s)^{\alpha-1}\bigl(\E[\|\Delta f\bigl({X}^{\epsilon}(t_{n_{\delta}(t)}),\widehat{Y}^{\epsilon,\delta}(s)\bigr)\|^2]\bigr)^{\frac12} \diff s\\
&\le C\int_{t_{n_\delta(s)}}^{t}(t-s)^{\alpha-1}\bigl(1+\E[\|{X}^{\epsilon}(t_{n_{\delta}(t)})\|^2] +\E[\|\widehat{Y}^{\epsilon,\delta}(s)\|^2]\bigr)^{\frac12} \diff s.
\end{align*}
Therefore, using the moment bounds~\eqref{eq:boundXepsYeps} from Proposition~\ref{propo:boundXepsYeps} for the process $X^\epsilon$  and the moment bounds~\eqref{eq:boundYaux} from Remark~\ref{rem:boundYaux} for the process $\widehat{Y}^{\epsilon,\delta}$, which are uniform with respect to the parameters $\epsilon,\delta\in(0,1)$, one obtains the upper bounds
\begin{align*}
\bigl(\E[&\|\int_{t_{n_\delta(s)}}^{t}(t-s)^{\alpha-1}\Delta f\bigl({X}^{\epsilon}(t_{n_{\delta}(t)}),\widehat{Y}^{\epsilon,\delta}(s)\bigr)\diff s\|^2]\bigr)^{\frac12}\\
&\le C(T)\bigl(1+\|x_0\|+\|y_0\|\bigr)\int_{t_{n_\delta(s)}}^{t}(t-s)^{\alpha-1} \diff s\\
&\le C(T)\bigl(1+\|x_0\|+\|y_0\|\bigr)\frac{(t-t_{n_\delta}(t))^\alpha}{\alpha}\\
&\le C_\alpha(T)\bigl(1+\|x_0\|+\|y_0\|\bigr)\delta^\alpha.
\end{align*}

Dealing with the second part of the error term $r_1^{\epsilon,\delta}(t)$ requires more attention. For all $k\in\{0,\ldots,n_{\delta}(t)-1\}$, set
\[
r_{1,k}^{\epsilon,\delta}(t)=\frac{1}{\Gamma(\alpha)}\int_{t_k}^{t_{k+1}}(t-s)^{\alpha-1}\Delta f\bigl({X}^{\epsilon}(t_k),\widehat{Y}^{\epsilon,\delta}(s)\bigr) \diff s.
\]
Then one has
\[
\E\Bigl[\Big\|\sum_{k=0}^{n_{\delta}(t)-1}r_{1,k}^{\epsilon,\delta}(t)\Big\|^2\Bigr]=\sum_{k=0}^{n_{\delta}(t)-1}\E\bigl[\big\|r_{1,k}^{\epsilon,\delta}(t)\big\|^2\bigr]+2\sum_{0\le k<\ell\le n_{\delta}(t)-1}\E\bigl[\langle r_{1,k}^{\epsilon,\delta}(t),r_{1,\ell}^{\epsilon,\delta}(t)\rangle\bigr].
\]
On the one hand, let $k\in\{0,\ldots,n_{\delta}(t)-1\}$, then one has
\begin{align*}
\E\bigl[&\big\|r_{1,k}^{\epsilon,\delta}(t)\big\|^2\bigr]\\
&=\frac{1}{\Gamma(\alpha)^2}\int_{t_k}^{t_{k+1}}\int_{t_k}^{t_{k+1}}(t-s_1)^{\alpha-1}(t-s_2)^{\alpha-1}\E\bigl[\langle \Delta f\bigl({X}^{\epsilon}(t_k),\widehat{Y}^{\epsilon,\delta}(s_1)\bigr),\Delta f\bigl({X}^{\epsilon}(t_k),\widehat{Y}^{\epsilon,\delta}(s_2)\bigr)\rangle\bigr] \diff s_1 \diff s_2\\
& =\frac{2}{\Gamma(\alpha)^2}\int_{t_k}^{t_{k+1}}\int_{s_1}^{t_{k+1}}(t-s_1)^{\alpha-1}(t-s_2)^{\alpha-1}\E\bigl[\langle \Delta f\bigl({X}^{\epsilon}(t_k),\widehat{Y}^{\epsilon,\delta}(s_1)\bigr),\Delta f\bigl({X}^{\epsilon}(t_k),\widehat{Y}^{\epsilon,\delta}(s_2)\bigr)\rangle\bigr] \diff s_1 \diff s_2.
\end{align*}
Given $s_2\ge s_1\ge t_k$, the random variables ${X}^{\epsilon}(t_k)$ and $\widehat{Y}^{\epsilon,\delta}(s_1)$ are $\mathcal{F}_{s_1}$-measurable. As a result, considering conditional expectation one has
\begin{align*}
\E\bigl[\langle \Delta f\bigl({X}^{\epsilon}(t_k),\widehat{Y}^{\epsilon,\delta}(s_1)\bigr)&,\Delta f\bigl({X}^{\epsilon}(t_k),\widehat{Y}^{\epsilon,\delta}(s_2)\bigr)\rangle\bigr]\\
&=\E\bigl[\langle \Delta f\bigl({X}^{\epsilon}(t_k),\widehat{Y}^{\epsilon,\delta}(s_1)\bigr),\E[\Delta f\bigl({X}^{\epsilon}(t_k),\widehat{Y}^{\epsilon,\delta}(s_2)\bigr)|\mathcal{F}_{s_1}]\rangle\bigr],
\end{align*}
and by the Markov property one has
\[
\E[\Delta f\bigl({X}^{\epsilon}(t_k),\widehat{Y}^{\epsilon,\delta}(s_2)\bigr)|\mathcal{F}_{s_1}]=\bigl(P_{\frac{s_2-s_1}{\epsilon}}^{{X}^{\epsilon}(t_k)}\Delta f({X}^{\epsilon}(t_k),\cdot)\bigr)\bigl(\widehat{Y}^{\epsilon,\delta}(s_1)\bigr),
\]
where the semigroup $\bigl(P_t^x\bigr)_{t\ge 0}$ with frozen slow component $x\in\R^p$ is given by~\eqref{eq:semigroupfrozenSDE}. Note that by construction one has $\int_{\R^q} \Delta f(x,y) d\mu^x(y)=0$ for all $x\in\R^p$. Moreover, the mapping $f$ is globally Lipschitz continuous, owing to Assumption~\ref{ass:F}. Therefore applying the upper bound~\eqref{eq:ergosemigroupfrozenSDE} one obtains the upper bound
\[
\big|\E[\Delta f\bigl({X}^{\epsilon}(t_k),\widehat{Y}^{\epsilon,\delta}(s_2)\bigr)|\mathcal{F}_{s_1}]\big|\le Ce^{-\frac{s_2-s_1}{\epsilon}}\bigl(1+\|{X}^{\epsilon}(t_k)\|+\|\widehat{Y}^{\epsilon,\delta}(s_1)\|\bigr).
\]
Since $f$ and $\overline{f}$ have at most linear growth (see~\eqref{eq:lineargrowth} and~\eqref{eq:lineargrowth2}), using the moment bounds~\eqref{eq:boundXepsYeps} from Proposition~\ref{propo:boundXepsYeps} for $X^\epsilon$ and the moment bounds~\eqref{eq:boundYaux} from Remark~\ref{rem:boundYaux} for $\widehat{Y}^{\epsilon,\delta}$, one obtains
\begin{align*}
\Big|\E\bigl[\langle \Delta f\bigl({X}^{\epsilon}(t_k),\widehat{Y}^{\epsilon,\delta}(s_1)\bigr)&,\Delta f\bigl({X}^{\epsilon}(t_k),\widehat{Y}^{\epsilon,\delta}(s_2)\bigr)\rangle\bigr]\Big|\\
&\le Ce^{-\frac{s_2-s_1}{\epsilon}}\E\bigl[\bigl(1+\|{X}^{\epsilon}(t_k)\|+\|\widehat{Y}^{\epsilon,\delta}(s_1)\|\bigr)^2\bigr]\\
&\le C(T)e^{-\frac{s_2-s_1}{\epsilon}}\bigl(1+\|x_0\|+\|y_0\|\bigr)^2.
\end{align*}
Therefore, for all $t\in[0,T]$ one obtains the upper bound
\[
\sum_{k=0}^{n_{\delta}(t)-1}\E\bigl[\big\|r_{1,k}^{\epsilon,\delta}(t)\big\|^2\bigr]\le \rho_1(t,\epsilon,\delta),
\]
where for all $t\in[0,T]$ the error term $\rho_1(t,\epsilon,\delta)$ is given by
\begin{equation}\label{eq:rho1}
\rho_1(t,\epsilon,\delta)=C(T)\sum_{k=0}^{n_{\delta}(t)-1}\int_{t_k}^{t_{k+1}}\int_{s_1}^{t_{k+1}}(t-s_1)^{\alpha-1}(t-s_2)^{\alpha-1} e^{-\frac{s_2-s_1}{\epsilon}} \diff s_1 \diff s_2 \bigl(1+\|x_0\|+\|y_0\|\bigr)^2.
\end{equation}

On the other hand, let $k,\ell\in\{0,\ldots,n_{\delta}(t)-1\}$ such that $k<\ell$, then one has
\begin{align*}
\E\bigl[&\langle r_{1,k}^{\epsilon,\delta}(t),r_{1,\ell}^{\epsilon,\delta}(t)\rangle\bigr]\\
&=\frac{1}{\Gamma(\alpha)^2}\int_{t_k}^{t_{k+1}}\int_{t_\ell}^{t_{\ell+1}}(t-s_1)^{\alpha-1}(t-s_2)^{\alpha-1}\E\bigl[\langle \Delta f\bigl({X}^{\epsilon}(t_k),\widehat{Y}^{\epsilon,\delta}(s_1)\bigr),\Delta f\bigl({X}^{\epsilon}(t_\ell),\widehat{Y}^{\epsilon,\delta}(s_2)\bigr)\rangle\bigr] \diff s_1 \diff s_2.
\end{align*}
Given $s_2\ge t_\ell \ge t_{k+1} \ge s_1\ge t_k$, the random variables ${X}^{\epsilon}(t_k)$, ${X}^{\epsilon}(t_\ell)$ and $\widehat{Y}^{\epsilon,\delta}(s_1)$ are $\mathcal{F}_{t_{\ell}}$-measurable. As a result, considering conditional expectation one has
\begin{align*}
\E\bigl[\langle \Delta f\bigl({X}^{\epsilon}(t_k)&,\widehat{Y}^{\epsilon,\delta}(s_1)\bigr),\Delta f\bigl({X}^{\epsilon}(t_\ell),\widehat{Y}^{\epsilon,\delta}(s_2)\bigr)\rangle\bigr]\\
&=\E\bigl[\langle \Delta f\bigl({X}^{\epsilon}(t_k),\widehat{Y}^{\epsilon,\delta}(s_1)\bigr),\E[\Delta f\bigl({X}^{\epsilon}(t_\ell),\widehat{Y}^{\epsilon,\delta}(s_2)\bigr)|\mathcal{F}_{t_\ell}]\rangle\bigr],
\end{align*}
and by the Markov property one has
\[
\E[\Delta f\bigl({X}^{\epsilon}(t_\ell),\widehat{Y}^{\epsilon,\delta}(s_2)\bigr)|\mathcal{F}_{t_\ell}]=\bigl(P_{\frac{s_2-t_\ell}{\epsilon}}^{{X}^{\epsilon}(t_\ell)}\Delta f({X}^{\epsilon}(t_\ell),\cdot)\bigr)\bigl(\widehat{Y}^{\epsilon,\delta}(t_\ell)\bigr),
\]
where the semigroup $\bigl(P_t^x\bigr)_{t\ge 0}$ with frozen slow component $x\in\R^p$ is given by~\eqref{eq:semigroupfrozenSDE}. Note that by construction one has $\int_{\R^q} \Delta f(x,y) d\mu^x(y)=0$ for all $x\in\R^p$. Moreover, the mapping $f$ is globally Lipschitz continuous, owing to Assumption~\ref{ass:F}. Therefore applying the upper bound~\eqref{eq:ergosemigroupfrozenSDE} one obtains the upper bound
\[
\big|\E[\Delta f\bigl({X}^{\epsilon}(t_\ell),\widehat{Y}^{\epsilon,\delta}(s_2)\bigr)|\mathcal{F}_{t_\ell}]\big|\le Ce^{-\frac{s_2-t_\ell}{\epsilon}}\bigl(1+\|{X}^{\epsilon}(t_\ell)\|+\|\widehat{Y}^{\epsilon,\delta}(t_\ell)\|\bigr).
\]
Since $f$ and $\overline{f}$ have at most linear growth (see~\eqref{eq:lineargrowth} and~\eqref{eq:lineargrowth2}), using the moment bounds~\eqref{eq:boundXepsYeps} from Proposition~\ref{propo:boundXepsYeps} for $X^\epsilon$ and the moment bounds~\eqref{eq:boundYaux} from Remark~\ref{rem:boundYaux} for $\widehat{Y}^{\epsilon,\delta}$, one obtains
\begin{align*}
\Big|\E\bigl[\langle \Delta f\bigl({X}^{\epsilon}(t_k)&,\widehat{Y}^{\epsilon,\delta}(s_1)\bigr),\Delta f\bigl({X}^{\epsilon}(t_\ell),\widehat{Y}^{\epsilon,\delta}(s_2)\bigr)\rangle\bigr]\Big|\\
&\le Ce^{-\frac{s_2-t_\ell}{\epsilon}}\E\bigl[\bigl(1+\|{X}^{\epsilon}(t_k)\|+\|\widehat{Y}^{\epsilon,\delta}(s_1)\|\bigr)\bigl(1+\|{X}^{\epsilon}(t_\ell)\|+\|\widehat{Y}^{\epsilon,\delta}(t_\ell)\|\bigr)\bigr]\\
&\le C(T)e^{-\frac{s_2-t_\ell}{\epsilon}}\bigl(1+\|x_0\|+\|y_0\|\bigr)^2.
\end{align*}
Therefore, for all $t\in[0,T]$ one obtains the upper bound
\[
2\sum_{0\le k<\ell\le n_{\delta}(t)-1}\E\bigl[\langle r_{1,k}^{\epsilon,\delta}(t),r_{1,\ell}^{\epsilon,\delta}(t)\rangle\bigr]\le \rho_2(t,\epsilon,\delta),
\]
where for all $t\in[0,T]$ the error term $\rho_2(t,\epsilon,\delta)$ is given by
\begin{equation}\label{eq:rho2}
\rho_2(t,\epsilon,\delta)=C(T)\sum_{0\le k<\ell\le n_{\delta}(t)-1}\int_{t_k}^{t_{k+1}}\int_{t_\ell}^{t_{\ell+1}}(t-s_1)^{\alpha-1}(t-s_2)^{\alpha-1} e^{-\frac{s_2-t_\ell}{\epsilon}}\diff s_1 \diff s_2 \bigl(1+\|x_0\|+\|y_0\|\bigr)^2.
\end{equation}
Let $\rho(t,\epsilon,\delta)=\rho_1(t,\epsilon,\delta)+\rho_2(t,\epsilon,\delta)$ for all $t\in[0,T]$.

Gathering the upper bounds, one obtains for all $t\in[0,T]$
\begin{equation}\label{eq:r1epsdelta}
\E[\|r_1^{\epsilon,\delta}(t)\|^2]\le C_\alpha(T,x_0,y_0)\delta^{2\alpha}+C(T,x_0,y_0)\rho(t,\epsilon,\delta).
\end{equation}

$\bullet$ Treatment of the error term $r_2^{\epsilon,\delta}(t)$.

The mapping $\overline{f}$ is globally Lipschitz continuous, therefore applying the Cauchy--Schwarz inequality and the inequality~\eqref{eq:errorX1} from Lemma~\ref{lem:errorX1}, for all $t\in[0,T]$ one has
\[
\E[\|r_2^{\epsilon,\delta}(t)\|^2]\le C_\alpha(T)\int_{0}^{t}(t-s)^{\alpha-1}\E\bigl[\|X^\epsilon(t_{n_\delta(s)})-\widehat{X}^{\epsilon,\delta}(t_{n_\delta(s)})\|^2\bigr]\diff s,
\]
and thus one obtains for all $t\in[0,T]$
\begin{equation}\label{eq:r2epsdelta}
\E[\|r_2^{\epsilon,\delta}(t)\|^2]\le C_\beta(T,x_0,y_0)\delta^{2\beta}.
\end{equation}

$\bullet$ Treatment of the error term $r_3^{\epsilon,\delta}(t)$.

The mapping $\overline{f}$ is globally Lipschitz continuous, therefore applying the Cauchy--Schwarz inequality,
one obtains for all $t\in[0,T]$
\begin{equation}\label{eq:r3epsdelta}
\E[\|r_3^{\epsilon,\delta}(t)\|^2]\le C_\beta(T)\int_{0}^{t}(t-s)^{\alpha-1}\E\bigl[\|\widehat{X}^{\epsilon,\delta}(t_{n_\delta(s)})-\overline{X}(t_{n_\delta(s)})\|^2\bigr] \diff s.
\end{equation}

$\bullet$ Treatment of the error term $r_4^{\epsilon,\delta}(t)$.

The mapping $\overline{f}$ is globally Lipschitz continuous, therefore applying the Cauchy--Schwarz inequality and the inequality~\eqref{eq:regulXbar} from Proposition~\ref{propo:regulXbar}, for all $t\in[0,T]$ one has
\[
\E[\|r_4^{\epsilon,\delta}(t)\|^2]\le C_\alpha(T)\int_{0}^{t}(t-s)^{\alpha-1}\E\bigl[\|\overline{X}(t_{n_\delta(s)})-\overline{X}(s)\|^2\bigr]\diff s,
\]
and thus one obtains for all $t\in[0,T]$
\begin{equation}\label{eq:r4epsdelta}
\E[\|r_4^{\epsilon,\delta}(t)\|^2]\le C_\beta(T,x_0,y_0)\delta^{2\beta}.
\end{equation}

$\bullet$ Conclusion.

Recalling the decomposition~\eqref{eq:decomperrorX2} of the error, gathering the upper bounds~\eqref{eq:r1epsdelta},~\eqref{eq:r2epsdelta},~\eqref{eq:r3epsdelta} and~\eqref{eq:r4epsdelta} obtained above, for all $t\in[0,T]$ one has
\begin{align*}
\E\bigl[\|\widehat{X}^{\epsilon,\delta}(t)-\overline{X}(t)\|^2\bigr]&\le C_\beta(T,x_0,y_0)\delta^{2\beta}+C(T,x_0,y_0)\rho(t,\epsilon,\delta)\\
&+C_\beta(T)\int_{0}^{t}(t-s)^{\alpha-1}\E\bigl[\|\widehat{X}^{\epsilon,\delta}(t_{n_\delta(s)})-\overline{X}(t_{n_\delta(s)})\|^2\bigr]\diff s.
\end{align*}
For any fixed auxiliary parameter $\delta$ and for all $t\in[0,T]$, as a consequence of the dominated convergence theorem, from~\eqref{eq:rho1} and~\eqref{eq:rho2} one has
\[
\underset{\epsilon\to 0}\lim~\rho(t,\epsilon,\delta)=0.
\]
For all $\delta\in(0,1)$ and $t\in[0,T]$, define
\begin{equation}\label{eq:erraux}
\overline{\rm err}^{\delta}(t)=\underset{\epsilon\to 0}\limsup~\E[\|\widehat{X}^{\epsilon,\delta}(t)-\overline{X}(t)\|^2].
\end{equation}
The upper bound above then gives
\[
\overline{\rm err}^{\delta}(t)\le C_\beta(T,x_0,y_0)\delta^{2\beta}+C_\beta(T)\int_{0}^{t}(t-s)^{\alpha-1}\overline{\rm err}^{\delta}(t_{n_\delta(s)})\diff s.
\]
Applying a version of the Gr\"onwall inequality shows that one has
\begin{equation}
\underset{t\in[0,T]}\sup~\overline{\rm err}^{\delta}(t)\le C_\beta(T,x_0,y_0)\delta^{2\beta}=C_\beta(T)\bigl(1+\|x_0\|+\|y_0\|)^2 \delta^{2\beta}.
\end{equation}
This yields the inequality~\eqref{eq:errorX2} and concludes the proof of Lemma~\ref{lem:errorX2}.
\end{proof}

\subsection{Proof of Theorem~\ref{theo:1}}

\begin{proof}[Proof of Theorem~\ref{theo:1}]
Recall the decomposition of the error given by~\eqref{eq:decomperror}. Combining the results of Lemma~\ref{lem:errorX1} and~\ref{lem:errorX2} yields
\[
\underset{t\in[0,T]}\sup~\underset{\epsilon\to 0}\limsup~\E[\|{X}^{\epsilon}(t)-\overline{X}(t)\|^2]\le C_\beta(T)(1+\|x_0\|+\|y_0\|)^2\delta^{2\beta}
\]
where the auxiliary parameter $\delta\in(0,1)$ is arbitrary, therefore letting $\delta\to 0$ gives
\[
\underset{t\in[0,T]}\sup~\underset{\epsilon\to 0}\limsup~\E[\|{X}^{\epsilon}(t)-\overline{X}(t)\|^2]=0.
\]
Thus for all $t\in[0,T]$ one obtains the convergence
\begin{equation}\label{eq:cv}
\underset{\epsilon\to 0}\lim~\E[\|{X}^{\epsilon}(t)-\overline{X}(t)\|^2]=0.
\end{equation}
Next, note that for all $t\in[0,T]$, one has
\begin{align*}
\bigl(\E[\|{X}^{\epsilon}(t)-\overline{X}(t)\|^2]\bigr)^{\frac12}&\le \bigl(\E[\|{X}^{\epsilon}(t)-{X}^{\epsilon}(t_{n_\delta(t)})\|^2]\bigr)^{\frac12}+\bigl(\E[\|\overline{X}(t_{n_\delta(t)})-\overline{X}(t)\|^2]\bigr)^{\frac12}\\
&+\bigl(\E[\|{X}^{\epsilon}(t_{n_\delta(t)})-\overline{X}(t_{n_\delta(t)})\|^2]\bigr)^{\frac12}.
\end{align*}
Let $\beta\in(0,\alpha)$. Owing to the bounds~\eqref{eq:regulXeps} and~\eqref{eq:regulXbar} from Propositions~\ref{propo:regulXeps} and~\ref{propo:regulXbar}, one has for all $t\in[0,T]$
\[
\bigl(\E[\|{X}^{\epsilon}(t)-{X}^{\epsilon}(t_{n_\delta(t)})\|^2]\bigr)^{\frac12}+\bigl(\E[\|\overline{X}(t_{n_\delta(t)})-\overline{X}(t)\|^2]\bigr)^{\frac12} \le C_\beta(T,x_0,y_0)\delta^{\beta}.
\]
As a consequence, one has
\begin{align*}
\underset{t\in[0,T]}\sup~\bigl(\E[\|{X}^{\epsilon}(t)-\overline{X}(t)\|^2]\bigr)^{\frac12}&\le C_\beta(T,x_0,y_0)\delta^{\beta}+\underset{n=0,\ldots,N}\sup~\bigl(\E[\|{X}^{\epsilon}(t_{n})-\overline{X}(t_{n})\|^2]\bigr)^{\frac12}\\
&\le C_\beta(T,x_0,y_0)\delta^{\beta}+\sum_{n=0}^{N}\bigl(\E[\|{X}^{\epsilon}(t_{n})-\overline{X}(t_{n})\|^2]\bigr)^{\frac12}.
\end{align*}
Owing to the convergence result~\eqref{eq:cv} above applied with $t\in\{t_n;~n=0,\ldots,N\}$, one obtains
\[
\underset{\epsilon\to 0}\limsup~\underset{t\in[0,T]}\sup~\bigl(\E[\|{X}^{\epsilon}(t)-\overline{X}(t)\|^2]\bigr)^{\frac12}\le C_\beta(T,x_0,y_0)\delta^{\beta}.
\]
Since the left-hand side of the upper bound above is independent of the auxiliary parameter $\delta$, letting $\delta\to 0$ yields
\[
\underset{\epsilon\to 0}\limsup~\underset{t\in[0,T]}\sup~\E[\|{X}^{\epsilon}(t)-\overline{X}(t)\|^2]=0.
\]
This concludes the proof of Theorem~\ref{theo:1}.
\end{proof}

\subsection{Proof of Theorem~\ref{theo:2}}

In this section, it is assumed that $b$ and $\sigma$ are independent of the slow component $x$, i.e. one has $b(x,y)=b(y)$ and $\sigma(x,y)=\sigma(y)$ for all $x\in\R^p$ and $y\in\R^q$, where for simplicity the same notation is used for mappings $b$ and $\sigma$ defined on $\R^q$. As a result, the fast process $\bigl(Y^\epsilon(t)\bigr)_{t\ge 0}$ is related to the stochastic differential equation
\[
\diff Y(t)=b(Y(t))\diff t+\sigma(Y(t))\diff B(t),\quad t\ge 0.
\]
which replaces the equation~\eqref{eq:frozenSDE} with frozen slow component. The invariant distribution is denoted by $\mu$. Instead of~\eqref{eq:semigroupfrozenSDE}, the associated semigroup is denoted by $\bigl(P_t\bigr)_{t\ge 0}$: for all $t\ge 0$, one has
\[
P_t\phi(y)=\E[\phi(Y(t))|Y(0)=y],\quad \forall~y\in\R^q,
\]
if $\phi:\R^q\to\R$ is a bounded and continuous mapping. Instead of~\eqref{eq:ergosemigroupfrozenSDE}, if $\phi$ is a Lipschitz continuous mapping, one has
\[
\big|P_t\phi(y)-\int_{\R^q}\phi(y) \diff\mu(y)\big|\le C{\rm Lip}(\phi)e^{-\gamma t}(1+\|y\|),\quad \forall~y\in\R^q,t\ge 0.
\]

\begin{proof}[Proof of Theorem~\ref{theo:2}]

Given~\eqref{eq:system} and~\eqref{eq:averaged}, the error is decomposed as follows: for all $t\in[0,T]$, one has
\begin{align*}
X^\epsilon(t)-\overline{X}(t)&=\frac{1}{\Gamma(\alpha)}\int_{0}^{t}(t-s)^{\alpha-1}\bigl[f(X^\epsilon(s),Y^\epsilon(s))-f(\overline{X}(s),Y^\epsilon(s))\bigr]\diff s\\
&+\frac{1}{\Gamma(\alpha)}\int_{0}^{t}(t-s)^{\alpha-1}\bigl[f(\overline{X}(s),Y^\epsilon(s))-\overline{f}(\overline{X}(s))\bigr]\diff s.
\end{align*}
Owing to Assumption~\ref{ass:F}, $f$ is globally Lipschitz continuous. Therefore applying the Cauchy--Schwarz inequality and using the auxiliary mapping $\Delta f$ defined by~\eqref{eq:Deltaf}, there exists $C_\alpha(T)\in(0,\infty)$ such that for all $t\in[0,T]$ one has
\begin{equation}\label{eq:boundproof2}
\E[\|X^\epsilon(t)-\overline{X}(t)\|^2]\le C_\alpha(T)\int_{0}^{t}(t-s)^{\alpha-1}\E[\|X^\epsilon(s)-\overline{X}(s)\|^2]\diff s+{\rm err}(t),
\end{equation}
where for all $t\in[0,T]$ one has
\begin{align*}
{\rm err}(t)&=\E\Bigl[\|\frac{1}{\Gamma(\alpha)}\int_{0}^{t}(t-s)^{\alpha-1}\bigl[f(\overline{X}(s),Y^\epsilon(s))-\overline{f}(\overline{X}(s))\bigr]\diff s\|^2\Bigr]\\
&=\frac{2}{\Gamma(\alpha)^2}\int_{0}^{t}\int_{s_1}^{t}(t-s_1)^{\alpha-1}(t-s_2)^{\alpha-1}\E\bigl[\langle \Delta f(\overline{X}(s_1),Y^\epsilon(s_1)),\Delta f(\overline{X}(s_2),Y^\epsilon(s_2))\rangle\bigr]\diff s_2 \diff s_1.
\end{align*}
Note that for all $s_2\ge s_1\ge 0$, the quantities $\overline{X}(s_1)$ and $\overline{X}(s_2)$ are deterministic. Moreover, the random variable $Y^\epsilon(s_1)$ is $\mathcal{F}_{s_1}$-measurable. Using the properties of the conditional expectation, one then obtains
\begin{align*}
\E\bigl[\langle \Delta f(\overline{X}(s_1),Y^\epsilon(s_1))&,\Delta f(\overline{X}(s_2),Y^\epsilon(s_2))\rangle\bigr]\\
&=\E\bigl[\langle \Delta f(\overline{X}(s_1),Y^\epsilon(s_1)),\E\bigl[\Delta f(\overline{X}(s_2),Y^\epsilon(s_2))|\mathcal{F}_{s_1}\bigr]\rangle\bigr].
\end{align*}
Using the semigroup $\bigl(P_t\bigr)_{t\ge 0}$ introduced above, applying the Markov property one obtains
\[
\E\bigl[\Delta f(\overline{X}(s_2),Y^\epsilon(s_2))|\mathcal{F}_{s_1}\bigr]=P_{\frac{s_2-s_1}{\epsilon}}\phi(\overline{X}(s_2),\cdot)(Y^\epsilon(s_1)).
\]
Since $\int_{\R^q}\Delta f(x,y)\diff y=0$ and since $y\in\R^q\mapsto \Delta f(x,y)$ is globally Lipschitz continuous, uniformly with respect to $x\in\R^p$ (by Assumption~\ref{ass:F}), applying the inequality above yields
\[
\big|P_{\frac{s_2-s_1}{\epsilon}}\phi(\overline{X}(s_2),\cdot)(Y^\epsilon(s_1))\big|\le C(1+\|Y^\epsilon(s_1)\|)e^{-\frac{\gamma(s_2-s_1)}{\epsilon}}.
\]
Using the upper bound
\[
\|\Delta f(x,y)\|\le C\bigl(1+\|x\|+\|y\|\bigr),\quad \forall~x\in\R^p,y\in\R^q,
\]
and the moment bounds~\eqref{eq:boundXepsYeps} from Proposition~\ref{propo:boundXbar} and the bounds~\eqref{eq:boundXbar} from Proposition~\ref{propo:boundXbar}, one then obtains for all $s_2\ge s_1\ge 0$
\[
\big|\E\bigl[\langle \Delta f(\overline{X}(s_1),Y^\epsilon(s_1)),\Delta f(\overline{X}(s_2),Y^\epsilon(s_2))\rangle\bigr]\big|\le C(T)\bigl(1+\|x_0\|^2+\|y_0\|^2\bigr)e^{-\frac{s_2-s_1}{\epsilon}}.
\]
As a consequence, for all $t\in[0,T]$ one has
\[
{\rm err}(t)\le C_\alpha(T)\bigl(1+\|x_0\|^2+\|y_0\|^2\bigr)\int_{0}^{t}\int_{s_1}^{t}(t-s_1)^{\alpha-1}(t-s_2)^{\alpha-1}e^{-\frac{\gamma(s_2-s_1)}{\epsilon}}\diff s_2 \diff s_1.
\]
For all $\epsilon\in(0,1)$ and $t\in[0,T]$, set
\begin{equation}\label{eq:I}
I_\alpha^\epsilon(t)=\int_{0}^{t}\int_{s_1}^{t}(t-s_1)^{\alpha-1}(t-s_2)^{\alpha-1}e^{-\frac{\gamma(s_2-s_1)}{\epsilon}}\diff s_2 \diff s_1.
\end{equation}
We claim that there exists $C_\alpha(T)\in(0,\infty)$ such that for all $\epsilon\in(0,1)$ one has
\begin{equation}\label{eq:claimintegral}
\underset{t\in[0,T]}\sup~I_\alpha^\epsilon(t)\le C_\alpha(T)\epsilon^{\alpha}.
\end{equation}
The claim~\eqref{eq:claimintegral} is obtained as follows. Performing a change of variable
\[
\left\lbrace
\begin{aligned}
&\theta=t-s_2\\
&z=s_2-s_1,
\end{aligned}
\right.
\]
one has
\begin{align*}
I_\alpha^\epsilon(t)&=\iint_{\R^2}(t-s_1)^{\alpha-1}(t-s_2)^{\alpha-1}e^{-\frac{\gamma(s_2-s_1)}{\epsilon}}\mathds{1}_{t\ge s_2\ge s_1\ge 0}\diff s_2 \diff s_1\\
&=\iint_{\R^2}(\theta+z)^{\alpha-1}\theta^{\alpha-1}e^{-\frac{\gamma z}{\epsilon}}\mathds{1}_{0\le \theta \le t}\mathds{1}_{z\ge 0}\mathds{1}_{\theta+z\le t}\diff\theta \diff z\\
&\le \int_{0}^{t}\int_{0}^{t}z^{\alpha-1}\theta^{\alpha-1}e^{-\frac{\gamma z}{\epsilon}}\diff\theta \diff z\\
&\le \frac{T^\alpha}{\alpha}\int_{0}^{\infty}z^{\alpha-1}e^{-\frac{\gamma z}{\epsilon}}\diff z.
\end{align*}
Finally, using the change of variable $z'=\gamma z/\epsilon$, one has
\[
\int_{0}^{\infty}z^{\alpha-1}e^{-\frac{\gamma z}{\epsilon}}\diff z=\frac{\epsilon^{\alpha}}{\gamma^{\alpha}}\int_{0}^{\infty}(z')^{\alpha-1}e^{-z'}\diff z'=\frac{\epsilon^{\alpha}\Gamma(\alpha)}{\gamma^{\alpha}},
\]
and this concludes the proof of the claim.

Combining the inequality~\eqref{eq:boundproof2} with~\eqref{eq:claimintegral} and applying the Gr\"onwall inequality, one finally obtains the upper bound
\[
\underset{t\in[0,T]}\sup~\E[\|X^\epsilon(t)-\overline{X}(t)\|^2]\le C_\alpha(T)\bigl(1+\|x_0\|^2+\|y_0\|^2\bigr)\epsilon^{\alpha}.
\]
This yields the inequality~\eqref{eq:theo2} and the proof of Theorem~\ref{theo:2} is completed.
\end{proof}

\begin{rem}
In the proof of Theorem~\ref{theo:2}, it is not required that the fast process $\bigl(Y^\epsilon(t)\bigr)_{t\ge 0}$ is solution to a stochastic differential equation. It would be sufficient to assume appropriate ergodicity properties and that convergence to the invariant distribution is exponentially fast.
\end{rem}

\end{document}